\documentclass[11pt]{article}
\usepackage{fontenc}
\usepackage{inputenc}
\usepackage{authblk}
\usepackage{mathtools}

\mathtoolsset{showonlyrefs}
\usepackage{amsmath}
\usepackage{amssymb}
\usepackage{amsthm}
\usepackage{subfigure}
\usepackage{caption2}
\usepackage{array}
\usepackage{xy}
\usepackage[pdftex]{graphicx}
\usepackage{hyperref}
\usepackage{color}
\usepackage{transparent}
\usepackage{latexsym}
\usepackage{float}
\usepackage{mathrsfs}
\usepackage{amssymb,url}

\makeatletter
\@addtoreset{equation}{section}
\makeatother

\newtheorem{lemma}{Lemma}[section] 
\newtheorem{theorem}[lemma]{Theorem} 
 
\newtheorem{claim}[lemma]{Claim} 
 
\newtheorem{definition}[lemma]{Definition}

\newtheorem{remark}[lemma]{Remark} 

\newcommand{\Vol}{{\mbox{Vol}}}

\usepackage{hyperref}

\usepackage{xcolor}
	\title{\bf Nonuniqueness of solutions to the $L_p$ chord Minkowski problem}
\begin{document}
	\author{  Yuanyuan Li
		\\  \small School of Mathematical Sciences, University of Science and Technology of China,\\
		\small Hefei, 230026, China \\ \small  E-mails:
		lyuanyuan@mail.ustc.edu.cn
	}
	\date{}
	\maketitle
	
	\begin{abstract}This paper explores the nonuniqueness of solutions to the $L_p$ chord Minkowski problem for negative $p.$ The $L_p$ chord Minkowski problem was recently posed by Lutwak, Xi, Yang and Zhang, which seeks to determine the necessary and sufficient conditions for a given finite Borel measure  such that it is the $L_p$ chord measure of a convex body, and it includes the chord Minkowski problem and the $L_p$ Minkowski problem. 
	\end{abstract}

	\baselineskip=16.4pt
	\parskip=3pt
	
	\section{\bf Introduction}
	\ \ \ \ \ 
	The central objects in study of convex geometry are convex bodies. A convex body in $n$-dimensional Euclidean space $\mathbb{R}^n,$ is a compact convex set with non-empty interior. The Brunn-Minkowski theory is a study of convex bodies which centers around the study of geometric functionals and the differential of these functionals. When geometric invariants arise as geometric functionals of convex bodies, geometric measures are often viewed as differentials of geometric invariants. One of the cornerstones of the Brunn-Minkowski theory is the Minkowski problem. It is a problem of priscribing geometric measure generated by convex bodies, which is concerned about necessary and sufficient conditions for a given measure such that it arises as the measure generated by a convex body. The most studied Minkowski-type problem is the classical Minkowski problem, which focuses on the surface area measures of convex bodies. For a comprehensive discussion on the Minkowski problem and its resolution, we recommend readers consulting  Pogorelov \cite{Pogorelov} and Cheng–Yau \cite{Cheng-Yau}.
	
	% When solution has good regularity, it's equivalent to solve a degenerate fully non-linear partial differential equation.
	Recently, a new family of geometric measures were introduced by Lutwak-XYZ\cite{LXZY2022} by studying of a variational formula regarding intergral geometric invariants of convex bodies called chord integrals. Let $K \in \mathcal{K}^n$ where $\mathcal{K}^n:=\{\text{all convex bodies in } \mathbb{R}^n\},$ the $q$th chord integral $I_q(K)$ is defined by
	\begin{equation}\label{integral}
	I_q(K)=\int_{\mathcal{L}^n}|K\cap\ell|^q d\ell,
	\end{equation}
	where $\mathcal{L}^n$ denotes the Grassmannian of 1-dimensional affine subspace of $\mathbb{R}^n,$ $|K\cap \ell|$ denotes the length of the chord $K\cap \ell,$ and the integration is with respect to Haar measure on the affine Grassmannian $\mathcal{L}^n,$ which is normalized to be a probability measure when restricted to rotations and to be $(n-1)$-dimensional Lebesgue measure when restricted to parallel translations.
	$$
	I_1(K)=V(k),\quad  I_0(K)=\frac{\omega_{n-1}}{n\omega_n}S(K), \quad I_{n+1}(K)=\frac{n+1}{\omega_n}V(K)^2,
	$$
	where $\omega_n$ denotes the volume of $n$-dimensional unit ball. Note that $I_q(B_n)=\frac{2^q\omega_n\omega_{n+q-1}}{\omega_{q}},$ where $B_n$ is the n-dimensional unit ball.  One can see from the above fomula that the chord integrals include the convex body's volume and surface area as two special cases. These are Crofton’s volume formula, Cauchy’s integral formula for surface area, and the Poincar\'{e}-Hadwiger formula,
	respectively (see [\cite{D. Ren}, \cite{Santalo}]).
	
	The chord measures and the Minkowski problems associated with chord measures were posed in \cite{LXZY2022}. They showed that the chord measures are the differentials of chord integrals and completely solved the chord Minkowski problem except for the critical case of the Christoffel-Minkowski problem. The $q$th chord measure is a finite Borel measure on $\mathbb{S}^{n-1}$ defined by
	\begin{equation}\label{measure1}
	F_q(K,\eta)=\frac{2q}{\omega_n}\int_{\nu^{-1}_K(\eta)}\tilde{V}_{q-1}(K,z)d \mathcal{H}^{n-1}(z),\text{   Borel }\eta\subset\mathbb{S}^{n-1},
	\end{equation}
	where $\widetilde{V}_{q-1}(K,z)$ is the $q-1$ th dual quermassintegral with respect to $z$.(See \eqref{dual}.)
	$$
	F_0(K,\cdot)=\frac{(n-1)\omega_{n-1}}{n\omega_n}S_{n-2}(K,\cdot), \quad F_1(K,\cdot)=S_{n-1}(K,\cdot),
	$$
	where $S_{i}(K,\cdot)$ is the $i$th order area measure of $K.$ Once chord measures are constructed, the $L_p$ chord measures follow naturally by extensions.
	For $K\in \mathcal{K}^n_{o}$ and $p\in \mathbb{R},$ the $L_p$ chord measures are defined by
	\begin{equation}\label{measure2}
	F_{p,q}(K,\eta)=\frac{2q}{\omega_n}\int_{\nu^{-1}_K(\eta)}(z\cdot\nu_{K}(z))^{1-p}\tilde{V}_{q-1}(K,z)d \mathcal{H}^{n-1}(z),\text{   Borel }\eta\subset\mathbb{S}^{n-1}.
	\end{equation}
	When $p=0,$ it is the cone-chord measure. When $q=1,$ $F_{p,1}(K,\cdot)$ is the $L_p$ surface area measure. When $q=0,$ $F_{p,0}(K,\cdot)$ is the $L_p-(n-2)$th area measure.
	
	The $L_p$-Minkowski problem was first formulated and studied by Lutwak in \cite{L1}. It has been rapidly attracting much attention; Lutwak introduced the important $L_p$ surface area measure and its associated Minkowski problem in the $L_p$ Brunn-Minkowski theory. Many cases of the $L_p$ Minkowski problem have been solved. The logarithmic Minkowski problem is one of the most central Minkowski type problems and is the problem of characterizing the cone-volume measure; see B$\ddot{o}$r$\ddot{o}$czky, Lutwak, Yang and Zhang [\cite{BLYZ13}, \cite{BLYZ12(5)}], Zhu [\cite{Zhu14},\cite{BHZ16}], Stancu [\cite{Stancu0240}, \cite{Stancu41}], Gage \cite{Gage18}, Xi and Leng \cite{XL43}, Firey \cite{Firey(17)}, Andrews \cite{Andrews (1)}, Chen, Huang, Li and Liu \cite{CHLL(10)}, Chen, Feng, Liu \cite{CFL 2022}, [\cite{H.Y45}, \cite{BCD(7)}] and reference therein. The centro-affine Minkowski problem is unsolved, see \cite{WZ06}. For more classical Brunn-Minkowski theory and its recent developments, we suggest readers to Schneider's book \cite{Schneider}.
	
	The $L_p$ chord Minkowski problem posed by Xi-LZY \cite{LXZY2022} is a problem of prescribing the $L_p$ chord measure: Given a finite Borel measure $\mu$ on $\mathbb{S}^{n-1},p\in \mathbb{R},$ and $q\geqslant 0.$ Asking what are the necessary and sufficient conditions for $\mu$ such that $\mu$ is the $L_p$ chord measure of a convex body $K\in \mathcal{K}^n_o,$ namely
	\begin{equation}\label{original eq}
	F_{p,q}(K,\cdot)=\mu
	\end{equation}
	when $p=1,$ it is the chord Minkowski problem. When $q=1,$ it is the $L_p$ Minkowski problem.
	When $\mu$ has a density $f$ that is an integrable nonnegtive function on $\mathbb{S}^{n-1},$ equation\eqref{original eq} becomes a new type of Monge-Amp\`{e}re equation on $\mathbb{S}^{n-1}$:
	\begin{equation}\label{MAeq}
	\mbox{det}(\nabla^2h+hI) =\frac{h^{p-1}f}{\widetilde{V}_{q-1}([h],\bar{\nabla} h)},\text{  on }\mathbb{S}^{n-1},
	\end{equation}
	where $\nabla^2h$ is the covariant differentiation of $h$ with respect to an orthonormal frame on $\mathbb{S}^{n-1},$ we look for a solution $h$ which is the support function for some nondegenerate convex body. We can extend $h$ to $\mathbb{R}^n$ via homogeneity and $\bar{\nabla} h$ is the Euclidean gradient of $h$ in $\mathbb{R}^n,$ and $\widetilde{V}_{q-1}([h],\bar{\nabla} h)$ is the $(q-1)$th dual quermassintegral of the Wulff-shape $[h]$ of $h$ with repect to the point $\bar{\nabla} h.$ 
	
	In their paper\cite{LXZY2022}, Lutwak-XYZ gave a sufficient condition for the symmetric case of the chord log-Minkowski problem by studying the delicate concentration properties of cone-chord measures. Shortly thereafter, Xi, Yang, Zhang and Zhao \cite{XYZZ2022} solved the $L_p$ chord Minkowski problem for $ p>1$ and for $0<p<1$ under the symmetric condition, where the origin symmetry played a crucial role in the case of $0 \leqslant p< 1.$ More recently, Xi, Guo and Zhao solved the $L_p$ chord Minkowski problem when $0 \leqslant p < 1$, without any symmetry assumptions. We solve the $L_p$ chord Minkowski problem when $ p < 0$, without any symmetry assumptions. Lately, we solved the $L_p$ chord Minkowski problem in the case of discrete measures whose supports are in general position for negative $p$ and $q>0.$ As for general Borel measure with a density, we also give a proof but need $p\in(-n,0)$ and $n+1>q\geqslant 1,$ without any symmetry assumptions.
	
	The aim of this paper is to establish some nonuniqueness for the $L_p$ chord Minkowski problem for $p<0<q.$ As far as we know, this is the first result towards uniqueness and nonuniqueness of solution to the $L_p$ chord Minkowski problem. Our main theorem is as following:
	\begin{theorem}\label{A}
		For $p<0,2<q<n+1,$ there exists a positive function $f\in C^{\infty}(\mathbb{S}^{n-1})$ such that \eqref{MAeq} admits at least two different solutions.
	\end{theorem}
	\begin{remark}
		We have to note that the method in this paper is inspired by X.-J. Wang et al\cite{JLW non} and Q.-R. Li et al\cite{LLL non}. But the situation in our case is more complicated, since the dual quermassintegral $\widetilde{V}_{q-1}(K,z)$ is a nonlocal term in the integrand of $I_q(K).$ And we also note that the $q-1$ th dual quermassintegral $\widetilde{V}_q(K, z)$ of $K$ with respect to $z\in \partial K$ is more delicate than the $q-1$ th dual quermassintegral $\widetilde{V}_q(K)$ of $K\in \mathcal{K}^n_o.$ 
	\end{remark}
	To prove the Theorem \ref{A}, we need to find at least two different solutions. One is constructed from the solution of classic Minkowski problem with a special right-hand-side, such that the solution to \eqref{MAeq} obtained in this way has its $q-$th chord integral as small as we want. The other solution is from a new existence result for equation \eqref{MAeq} by the variational method. Before we state the result, we first need to introduce some notations.
	\begin{definition}
		A function $f:\mathbb{S}^{n-1}\rightarrow \mathbb{R}$ is called rotationally symmetric if it satisfies
		$$
		f(Ax', x_n)=f(x', x_n),\quad \forall x=(x',x_n)\in \mathbb{S}^{n-1}\text{ and }A\in O(n-1),
		$$
		where $x'=(x_1,\cdots,x_{n-1})$ and $O(\cdot)$ denotes the orthogonal group.
	\end{definition}
	Then, here comes our new existence result of solutions to \eqref{MAeq}.
	\begin{theorem}\label{B}
		For $p<0,1\leqslant q<n+1,$ and $\alpha,\beta$ satisfying
		$$
		\begin{gathered}
		\alpha>\max\{1-n, 1-n+\frac{2-n-q}{n+q-1}p\},\\
		\beta>\max\{-1,-1-\frac{p}{n+q-1}\}.
		\end{gathered}
		$$
		If $f$ is a nonnegative, rotationally symmetric, even function on $\mathbb{S}^{n-1}$ and satisfies
		\begin{equation}\label{f}
		f\leqslant C\left|x^{\prime}\right|^\alpha\left|x_n\right|^\beta,\|f\|_{L^1(\mathbb{S}^{n-1})}>0
		\end{equation}
		for some positive constant $C.$ Then there exists a rotationally symmetric even solution to \eqref{MAeq}. Moreover, we have its chord integral as follows:
		\begin{equation}\label{variational integral}
		I_q(K)\geqslant c>0,
		\end{equation}
		where $c$ depends only on $n,p,q,\alpha,\beta.$
	\end{theorem}
	\begin{remark}
		In our previous paper, we solved the $L_p$ chord Minkowski problem for general Borel measure with a density $f$, without any symmetry assumptions , but need $p\in(-n,0)$ and $n+1>q\geqslant 1.$ Here, when $f$ satisfies the conditions in Theorem \ref{B}, we can obtain the existence result for all $p<0.$ While the proof of the two existence results follows similar approaches, we would like to emphasize that the proof of $C^0$ estimate is totally different and more difficult. 
	\end{remark}
	
	The remainder of this paper is structured as follows. In Section 2, we present fundamental concepts in the theory of convex bodies and integral geometry. In Section 3, we will construct a smooth, positive function $f,$ which is rotationally symmetric in $x_1,\cdots, x_{n-1}$ and even, such that \eqref{MAeq} has a solution with small chord integral. In Section 4, we establish a variational solution and prove Theorem \ref{B}. In Section 5, we prove Theorem \ref{A} based on the existence result in Theorem \ref{B} and the solution constructed in Section 3.
	
	\section{\bf Preliminaries}
	In this section, our objective is to establish notations and gather relevant results from the literature that will be necessary for the subsequent analysis. 
	
	We denote $x\cdot y$ as the standard inner product of $x,y\in \mathbb{R}^n,$ and write $|x|=\sqrt{x\cdot x}$ for the Euclidean norm of $x.$ We write $\mathbb{S}^{n-1}$ as $(n-1)$-dimension unit sphere of $\mathbb{R}^n,$ and denote $\mathcal{H}^{n-1}$ as the $(n-1)$-dimensional spherical Lebesgue measure. Denote $\mathcal{K}^n$ for the collection of all convex bodies in $\mathbb{R}^n$ and $\mathcal{K}^n_o$ for the subset of $\mathcal{K}^n$ that contains the origin in the interior.
	
	Let $\Omega\subset \mathbb{S}^{n-1}$ be a closed set of the unit sphere, not lying in a closed hemisphere, and a positive continuous function $h: \mathbb{S}^{n-1}\rightarrow \mathbb{R}$ is given.(Only the values of h on $\Omega$ will be needed, but without loss of generality we may assume that $h$ is defined on all of $\mathbb{S}^{n-1}.$) The Wulff shape of $h$ is defined by
	$$
	[h]=\{x\in\mathbb{R}^n: x\cdot u\leqslant h(u)\text{ for all } u\in\mathbb{S}^{n-1}\}.
	$$

	Let $K\in \mathcal{K}^n,$ $h(v)=h_K(v)=\max\{v\cdot x,x\in K\},$ $\rho(u)=\rho_K(u)=\max\{\lambda:\lambda u\in K\}$ are the support function and the radial function of convex body $K$ defined from $ \mathbb{S}^{n-1}\rightarrow \mathbb{R}.$ We write the support hyperplane of $K$ with the outer unit normal $v$ as
	$$
	H_K(v)=\left\{x\in \mathbb{R}^n: x\cdot v=h(v) \right\},
	$$
	the half-space $H^{-}(K,v)$ in direction $v$ is defined by
	$$
	H^{-}_K(v)=\left\{x\in \mathbb{R}^n: x\cdot v\leqslant h(v) \right\}.
	$$  
	Denote $\partial K$ as the boundary of $K$, that is, $\partial K=\{\rho_K(u)u:u\in \mathbb{S}^{n-1}\}.$ The spherical image $\nu=\nu_{K}:\partial K\rightarrow \mathbb{S}^{n-1}$ is given by
	$$
	\nu(x)=\{v\in\mathbb{S}^{n-1}:x\in H_K(v)\},
	$$
	let $\sigma_K\subset \partial K$ denote the set of all points $x\in \partial K,$ such that the set $\nu_K(x)$ contains more than one element. Fortunately, we have $\mathcal{H}^{n-1}(\sigma_K)=0$ (see \cite[page 84]{Schneider} ) 
	and the radial Gauss image $\alpha=\alpha_K$ and the reverse radial Gauss image $\alpha^*=\alpha^*_K$ are respectively defined by
	$$
	\alpha(\omega)=\{\nu(\rho_K(u)u):u\in \omega\},\alpha^*(\omega)=\{u\in \mathbb{S}^{n-1} \nu(\rho_K(u)u)\in \omega\}.
	$$
	Let $K \in \mathcal{K}^n$, for $z \in \operatorname{int} K$ and $q \in \mathbb{R}$, the $q$ th dual quermassintegral $\widetilde{V}_q(K, z)$ of $K$ with respect to $z$ is defined by
	\begin{equation}\label{dual}
	\widetilde{V}_q(K, z)=\frac{1}{n} \int_{S^{n-1}} \rho_{K, z}(u)^q \mathrm{~d} u
	\end{equation}
	where $\rho_{K, z}(u)=\max \{\lambda>0: z+\lambda u \in K\}$ is the radial function of $K$ with respect to $z$. When $z \in \partial K, \widetilde{V}_q(K, z)$ is defined in the way that the integral is only over those $u \in S^{n-1}$ such that $\rho_{K, z}(u)>0$. In another word,
	$$
	\widetilde{V}_q(K, z)=\frac{1}{n} \int_{\rho_{K, z}(u)>0} \rho_{K, z}(u)^q \mathrm{~d} u \text {, whenever } z \in \partial K .
	$$
	In this case, for $\mathcal{H}^{n-1}$-almost all $z \in \partial K$, we have
	$$
	\widetilde{V}_q(K, z)=\frac{1}{2 n} \int_{S^{n-1}} X_K(z, u)^q \mathrm{~d} u
	$$
	where the parallel $X$-ray of $K$ is the nonnegative function on $\mathbb{R}^n \times S^{n-1}$ defined by
	$$
	X_K(z, u)=|K \cap(z+\mathbb{R} u)|, \quad z \in \mathbb{R}^n, \quad u \in S^{n-1} .
	$$
	When $q>0$, the dual quermassintegral is the Riesz potential of the characteristic function, that is,
	$$
	\widetilde{V}_q(K, z)=\frac{q}{n} \int_K|x-z|^{q-n} \mathrm{~d} x
	$$
	Note that this immediately allows for an extension of $\widetilde{V}_q(K, \cdot)$ to $\mathbb{R}^n$. An equivalent definition via radial function can be found in \cite{LXZY2022}. By a change of variables, we obtain:
	$$
	\widetilde{V}_q(K, z)=\frac{q}{n} \int_{K-z}|y|^{q-n} \mathrm{~d} y
	$$
	since when $q>0$, the integrand $|y|^{q-n}$ being locally integrable, it can be inferred that the dual quermassintegral $\widetilde{V}_q(K, z)$ is continuous in $z$. Let $K \in \mathcal{K}^n$. The $X$-ray $X_K(x, u)$ and the radial function $\rho_{K, z}(u)$ are related as follows:
	$$
	X_K(x, u)=\rho_{K, z}(u)+\rho_{K, z}(-u), \quad \text { when } \quad K \cap(x+\mathbb{R} u)=K \cap(z+\mathbb{R} u) \neq \varnothing .
	$$
	When $z \in \partial K$, then either $\rho_{K, z}(u)=0$ or $\rho_{K, z}(-u)=0$ for almost all $u \in S^{n-1}$, and thus
	$$
	X_K(z, u)=\rho_{K, z}(u), \quad \text { or } X_K(z, u)=\rho_{K, z}(-u), \quad z \in \partial K \text {, }
	$$
	for almost all $u \in S^{n-1}$. Then, the chord integral $I_q(K)$ can be represented as follows:
	$$
	I_q(K)=\frac{1}{n \omega_n} \int_{S^{n-1}} \int_{u^{\bot}} X_K(x, u)^q \mathrm{~d} x \mathrm{~d} u, \quad q \geq 0 .
	$$
	An elementary property of the functional $I_q$ is its homogeneity. If $K \in \mathcal{K}^n$ and $q \geq 0$, then
	$$
	I_q(t K)=t^{n+q-1} I_q(K),
	$$
	for $t>0$. By compactness of $K$, it is simple to see that the chord integral $I_q(K)$ is finite whenever $q \geq 0$.
	Let $K \in \mathcal{K}^n$ and $q>0$, the chord measure $F_q(K, \cdot)$ is a finite Borel measure on $S^{n-1}$, which can be expressed as:
	$$
	F_q(K, \eta)=\frac{2 q}{\omega_n} \int_{v^{-1}(\eta)} \widetilde{V}_{q-1}(K, z) \mathrm{d} \mathcal{H}^{n-1}(z), \quad \text { for each Borel } \eta \subset S^{n-1}.
	$$
	The mapping $v_K$ of $K$ is almost everywhere defined on $\partial K$ with respect to the $(n-1)$-dimensional Hausdorff measure, owing to the convexity of $K$. The chord measure $F_q(K, \cdot)$ is significant as it is obtained by differentiating the chord integral $I_q$ in a certain sense, as shown in \eqref{diff}. It is evident that the chord measure $F_q(K, \cdot)$ is absolutely continuous with respect to the surface area measure $S_{n-1}(K, \cdot)$. In \cite[Theorem 4.3]{LXZY2022}, it was demonstrated that:
	$$
	I_q(K)=\frac{1}{n+q-1} \int_{s^{n-1}} h_K(v) \mathrm{d} F_q(K, v)
	$$
	When $q>0$, a useful integral formula demonstrated in \cite[Lemma 5.3]{LXZY2022} is
	\begin{footnotesize}
		$$
		2 n \int_{\partial K} \widetilde{V}_{q-1}(K, z) g\left(v_K(z)\right) \mathrm{d} \mathcal{H}^{n-1}(z)=\int_{s^{n-1}} \int_{\partial K} X_K(z, u)^{q-1} g\left(v_K(z)\right) \mathrm{d} \mathcal{H}^{n-1}(z) \mathrm{d} u,
		$$
	\end{footnotesize}
	for any $g \in C\left(S^{n-1}\right)$. Therefore, for each $K \in \mathcal{K}^n$, we have
	$$
	\begin{aligned}
	\int_{S^{n-1}} g(v) \mathrm{d} F_q(K, v) & =\frac{q}{n \omega_n} \int_{S^{n-1}} \int_{\partial K} X_K(z, u)^{q-1} g\left(v_K(z)\right) \mathrm{d} \mathcal{H}^{n-1}(z) \mathrm{d} u \\
	& =\frac{q}{n \omega_n} \int_{S^{n-1}} \int_{S^{n-1}} X_K\left(\rho_K(w) w, u\right)^{q-1} h_K\left(\alpha_K(w)\right)^{-1}\\
	&\quad \rho_K(w)^n g\left(\alpha_K(w)\right) \mathrm{d} w \mathrm{~d} u .
	\end{aligned}
	$$
	Here, we denote $\rho_K=\rho_{K, o}$. For each $p \in \mathbb{R}$ and $K \in \mathcal{K}_o^n$, the $L_p$ chord measure $F_{p, q}(K, \cdot)$ is defined as follows:
	$$
	\mathrm{d} F_{p, q}(K, v)=h_K(v)^{1-p} \mathrm{~d} F_q(K, v)
	$$
	and we have an important property of $F_{p,q},$ its homogeneity, namely
	\begin{equation*}
	F_{p,q}(tK,\cdot)=t^{n+q-p-1}F_{p,q}(K,\cdot)
	\end{equation*}
	for each $t>0.$
	
	From Theorem 2.2 in \cite{XYZZ2022}, we know that if $K_i\in \mathcal{K}^{n}_o \rightarrow K_0\in \mathcal{K}^{n}_o,$ then the chord measure $F_q(K_i,\cdot)$ converges to $F_q(K,\cdot)$ weakly. Hence, one can immediately obtain that 
	$$
	F_{p,q}(K_i,\cdot) \rightarrow F_{p,q}(K,\cdot) \text{ weakly. }
	$$
	It was shown in \cite{LXZY2022} that the differential of the chord integral $I_q$ with respect to the $L_p$ Minkowski combinations leads to the $L_p$ chord measure: for $p \neq 0$,
	$$
	\left.\frac{\mathrm{d}}{\mathrm{d} t}\right|_{t=0} I_q\left(K+_p t \cdot L\right)=\frac{1}{p} \int_{S^{n-1}} h_L^p(v) \mathrm{d} F_{p, q}(K, v),
	$$
	where $K+_p t \cdot L$ is the $L_p$ Minkowski combination between $K$ and $L.$
	
	Since we will use a variational method to solve the $L_p$ chord Minkowski problem, the variational formula for chord integral is crucial and it is the key to tansforming the Minkowski problem into the Lagrange equation of an optimization problem. 
	\begin{theorem}[Theorem 5.5 in \cite{LXZY2022}]
		Let $q>0,$ and $\Omega$ be a compact subset of $\mathbb{S}^{n-1}$ that is not contained in any closed hemisphere. Suppose that $g:\Omega\rightarrow(0,\infty)$ is a family of continuous functions given by
		$$
		h_t=h_0+tg+o(t,\cdot),
		$$
		for each $t\in(-\delta,\delta)$ for some $\delta>0.$ Here $o(t,\cdot)\in C(\Omega)$ and $o(t,\cdot)/v$ tends to $0$ uniformly on $\Omega$ as $t\rightarrow 0.$ Let $K_t$ be the Wulff shape generated by $h_t$ and $K$ be the Wulff shape generated by $h_0.$ Then,
		\begin{equation}\label{diff}
		\frac{d}{dt}\big|_{t=0}I_q(K_t)=\int_{\Omega}g(v)dF_q(K,v).
		\end{equation}
	\end{theorem}
	See also in \cite[Theorem 2.1]{XYZZ2022}.
	
	To solve the maximization problem posed in Section 4, delicate estimates for chord integrals are needed. We collect the following lemma obtained in \cite{LXZY2022}.
	\begin{lemma}[lemma 7.3 \cite{LXZY2022}]\label{elliptic estimate}
		Suppose $q \in(1, n+1)$ is not an integer. If $E$ is the ellipsoid in $\mathbb{R}^n$ given by
		$$
		E=E\left(a_1, \ldots, a_n\right)=\left\{x \in \mathbb{R}^n: \frac{\left(x \cdot e_1\right)^2}{a_1^2}+\cdots+\frac{\left(x \cdot e_n\right)^2}{a_n^2} \leq 1\right\}
		$$
		with $0<a_1 \leq a_2 \leq \cdots \leq a_n \leq 1$, then for any real $q$ and integer $m$ such that $1 \leq m<q<$ $m+1 \leq n+1$,
		$$
		I_q(E) \leq c_{q, m, n}\left(a_1 \cdots a_m\right)^2 a_m^{q-m-1} a_{m+1} \cdots a_n
		$$
		where $c_{q, m, n}$ is a constant that depends only on $q$ and $n$ (since $m=\lfloor q\rfloor$ ) and is given by
		$$
		c_{q, m, n}= \begin{cases}\frac{2^{q-n+2} q(q-1) \omega_{n-1}^2}{(q-n)(q-n+1) n \omega_n} & m=n \\ \frac{2^{n-m+3} q(q-1)(n-m) \omega_{m-1}^2 \omega_{n-m}^2}{(m+1-q)(q-m)(q-m+1) n \omega_n} & m<n .\end{cases}
		$$
	\end{lemma}
	
	As for $q \in \{1,\cdots, n\},$ we shall deduce the same form of estimate. First, we recall a significant inequality presented in the following lemma, which was obtained in \cite{LXZY2022}:
	\begin{lemma}[claim 8.1 \cite{LXZY2022}]\label{chord integral ineq}
		If $K\in \mathcal{K}^n_o$ and $1\leqslant r<s,$ then 
		$$
		I_r(K)\leqslant c(s,r)V(K)^{1-\frac{r-1}{s-1}}I_s(K)^{\frac{r-1}{s-1}},
		$$ 
		with $c(s,r)=rs^{-\frac{r-1}{s-1}}.$
	\end{lemma}
	This inequality follows from a simple argument using jensen's inequality. And one immediately obtains the desired estimate for chord integrals when $q$ is an integer.  
	
	\section{\bf Structural solution }
	Let $0<\epsilon<\frac{1}{2}$, $M_{\epsilon}\in GL(n)$ be given by
	$$
	M_{\epsilon}=diag(\epsilon,\cdots,\epsilon,1)=\left(\begin{array}{cc}
	\epsilon I & 0 \\
	0 & 1
	\end{array}\right),
	$$
	where $I$ is the unit $(n-1)\times (n-1)$ matrix.
	
	Consider the following equation:
	\begin{equation}\label{Minkowski}
	\operatorname{det}\left(\nabla^2 h+h I\right)(x)=\left|x^{\prime}\right|^\alpha\left|x_n\right|^\beta\left|M_\epsilon x\right|^{\gamma-\beta}, \quad x \in \mathbb{S}^{n-1} .
	\end{equation}
	
	We can choose appropriate indices $\alpha$, $\beta$ and $\gamma,$ in this section, we need $2<q\leqslant 1+n,-1<\gamma<-1-\frac{p}{n+q-1},$ $\alpha$ and $\beta$ are nonnegative and satisfying the assumptions of Theorem \ref{B}. As its right-hand side is even with respect to the origin and satisfies the necessary condition 
	$$
	\int_{S^{n-1}}x_k\left|x^{\prime}\right|^\alpha\left|x_n\right|^\beta\left|M_\epsilon x\right|^{\gamma-\beta}=0
	$$ 
	for all $1\leqslant k\leqslant n$, this classical Minkowski problem exists a solution $h_\epsilon$, which is unique up to translation by \cite{Cheng-Yau}.
	
	Let $h_\epsilon$ be the unique solution such that its associated convex body $K_{h_\epsilon}$ centred at the origin. We have to note that $K_{h_\epsilon}$ is rotationally symmetric and even and the positive constants $C, \tilde{C}$ and $c_i,C_i$ in the following context depend only on $n, p, q, \alpha, \beta$, and $\gamma$ but independent of $\epsilon$.
	
	\begin{lemma}
		Let $\gamma>-1.$ There exists a positive constant $C$, independent of $\epsilon \in(0,\frac{1}{2})$, such that
		\begin{equation}\label{h epsilon bound}
		C^{-1} \leq h_\epsilon \leq C \text { on } \quad \mathbb{S}^{n-1}.
		\end{equation}
	\end{lemma}
	\begin{proof}
		By \eqref{Minkowski}, $-1<\gamma,$ one has
		$$
		\begin{aligned}
		\operatorname{area}\left(\partial K_{h_\epsilon}\right) 
		&=\int_{\mathbb{S}^{n-1}}\operatorname{det}\left(\nabla^2 h+h I\right)(x)dx\\
		&=\int_{\mathbb{S}^{n-1}}\left|x^{\prime}\right|^\alpha\left|x_n\right|^\beta\left(\epsilon^2\left|x^{\prime}\right|^2+\left|x_n\right|^2\right)^{\frac{\gamma-\beta}{2}}dx\\
		& \leqslant \int_{\mathbb{S}^{n-1}}\left(\epsilon^2\left|x^{\prime}\right|^2+\left|x_n\right|^2\right)^{\gamma/2}dx\\
		& \leqslant\left\{\begin{array}{lll}
		\int_{\mathbb{S}^{n-1}}\left|x_n\right|^{\gamma/2} \mathrm{~d} x & \text { when } -1<\gamma<0 \\
		\int_{\mathbb{S}^{n-1}} \mathrm{~d} x & \text { when } 0\leqslant\gamma
		\end{array}\right. \\
		& \leq C,
		\end{aligned}
		$$
		where $C$ is a positive constant depending on $n, \gamma$ but independent of $\epsilon$. 
		By John's lemma, 
		$$
		\frac{1}{n} E_\epsilon \subset K_{h_\epsilon} \subset E_\epsilon
		$$
		where $E_\epsilon$ is the minimum ellipsoid of $K_{h_\epsilon}$. Then, this implies that
		$$
		\frac{1}{n} h_{E_\epsilon} \leq h_\epsilon \leq h_{E_\epsilon} \quad \text { on } \quad \mathbb{S}^{n-1},
		$$
		where $h_{E_\epsilon}$ is the support function of $E_\epsilon$. Since $K_{h_\epsilon}$ is rotationally symmetric and even, we have that $E_\epsilon$ is also rotationally symmetric and even and centred at the origin. Let $r_{1 \epsilon}, \cdots, r_{n \epsilon}$ be the lengths of the semi-axes of $E_\epsilon$ along the $x_1, \cdots, x_n$ axes. Then, $r_{1 \epsilon}=\cdots=r_{n-1 ; \epsilon}$ and
		$$
		h_{E_\epsilon}(x)=\sqrt{r_{1 \epsilon}^2\left|x^{\prime}\right|^2+r_{n \epsilon}^2 x_n^2}, \quad \forall x \in \mathbb{S}^{n-1} .
		$$
		By \eqref{Minkowski}, one has
		$$
		\begin{aligned}
		\Vol(K_{h_\epsilon})&=\frac{1}{n}\int_{\mathbb{S}^{n-1}}h_\epsilon \operatorname{det}\left(\nabla^2 h_\epsilon+h_\epsilon I\right)(x)dx\\
		&=\frac{1}{n}\int_{\mathbb{S}^{n-1}}h_\epsilon\left|x^{\prime}\right|^\alpha\left|x_n\right|^\beta\left(\epsilon^2\left|x^{\prime}\right|^2+\left|x_n\right|^2\right)^{\frac{\gamma-\beta}{2}}dx\\
		&\geqslant \frac{1}{n}\int_{\mathbb{S}^{n-1}}h_\epsilon\left|x^{\prime}\right|^\alpha\left|x_n\right|^\beta\left(\epsilon^2\left|x^{\prime}\right|^2+\left|x_n\right|^2\right)^{\frac{\gamma}{2}}dx\\
		& \geqslant \frac{1}{n}\left\{\begin{array}{lll}
		\int_{\mathbb{S}^{n-1}} h_\epsilon(x)\left|x^{\prime}\right|^\alpha\left|x_n\right|^\beta \mathrm{d} x & \text { when } & -1<\gamma<0, \\
		\int_{\mathbb{S}^{n-1}} h_\epsilon(x)\left|x^{\prime}\right|^\alpha\left|x_n\right|^{\beta+\gamma} \mathrm{~d} x & \text { when } & \gamma\geqslant 0 .
		\end{array}\right. \\
		\end{aligned}
		$$
		Since
		$$
		h_\epsilon(x) \geq \frac{1}{\sqrt{2} n}\left(r_{1 \epsilon}\left|x^{\prime}\right|+r_{n \epsilon}\left|x_n\right|\right), \quad \forall x \in \mathbb{S}^{n-1}.
		$$
		Recall that $\alpha,\beta$ are nonnegative, then we have
		$$
		\begin{aligned}
		\Vol(K_{h_\epsilon})
		& \geqslant \frac{1}{n}\left\{\begin{array}{lll}
		\int_{\mathbb{S}^{n-1}} \frac{1}{\sqrt{2} n}\left(r_{1 \epsilon}\left|x^{\prime}\right|+r_{n \epsilon}\left|x_n\right|\right)\left|x^{\prime}\right|^\alpha\left|x_n\right|^\beta \mathrm{d} x & \text { when } & -1<\gamma<0, \\
		\int_{\mathbb{S}^{n-1}} \frac{1}{\sqrt{2} n}\left(r_{1 \epsilon}\left|x^{\prime}\right|+r_{n \epsilon}\left|x_n\right|\right)\left|x^{\prime}\right|^\alpha\left|x_n\right|^{\beta+\gamma} \mathrm{~d} x & \text { when } & \gamma\geqslant 0 .
		\end{array}\right. \\
		&\geqslant c (r_{1 \epsilon}+r_{n \epsilon}),
		\end{aligned}
		$$
		therefore, by the isoperimetric inequality, we obtain that
		$$
		\operatorname{vol}\left(K_{h_\epsilon}\right) \leq C_n \operatorname{area}\left(\partial K_{h_\epsilon}\right)^{\frac{n}{n-1}} \leq C ,
		$$
		and hence
		\begin{equation}\label{up}
		\max h_\epsilon<r_{1 \epsilon}+r_{n \epsilon}\leqslant c\Vol(K_{h_\epsilon})\leqslant C.
		\end{equation}
		On the other hand, we have
		$$
		\begin{aligned}
		\operatorname{vol}\left(K_{h_\epsilon}\right) & \leq \operatorname{vol}\left(E_\epsilon\right) \\
		& =\kappa_n r_{1 \epsilon}^{n-1} r_{n \epsilon} \\
		& \leq C_n\left(\max h_\epsilon\right)^{n-1} \cdot \min h_\epsilon
		\end{aligned}
		$$
		where $\kappa_n$ is the volume of the unit ball in $\mathbb{R}^n$.
		Since  
		$$
		\max h_\epsilon \leqslant c\Vol(K_{h_\epsilon}) \leq C\left(\max h_\epsilon\right)^{n-1} \cdot \min h_\epsilon
		$$
		namely
		$$
		1 \leq C\left(\max h_\epsilon\right)^{n-2} \cdot \min h_\epsilon
		$$
		By \eqref{up}, we obtain the bound of $h_\epsilon$ from below.
		Now, the proof of this lemma is completed.
	\end{proof}
	
	Define
	\begin{equation}\label{H epsilon}
	H_\epsilon(x):=\epsilon^{\frac{n-p-4-\gamma+q}{n-p+q-1}}\left|M_\epsilon^{-1} x\right| \cdot h_\epsilon\left(\frac{M_\epsilon^{-1} x}{\left|M_\epsilon^{-1} x\right|}\right), \quad x \in \mathbb{S}^{n-1} .
	\end{equation}
	\begin{lemma}
		The function $H_\epsilon$ satisfies the equation
		\begin{equation}\label{epsilon EQ}
		\mbox{det}(\nabla^2H_\epsilon+H_\epsilon) =\frac{H_\epsilon^{p-1}f_\epsilon}{\widetilde{V}_{q-1}([H_\epsilon],\bar{\nabla} H_\epsilon)},\text{  on }\mathbb{S}^{n-1},
		\end{equation} 
		where
		\begin{equation}
		f_\epsilon(x):=h_\epsilon\left(x_\epsilon\right)^{1-p}\left|X^{\prime}\right|^\alpha\left|x_n\right|^\beta\left|N_\epsilon X\right|^{-\gamma-\alpha-n-p}\frac{1}{n}\int_{S^{n-1}}\left|N_\epsilon y\right|^{q-1-n}\rho_{K_{h_\epsilon},\bar{\nabla}h_{\epsilon}(x_\epsilon)}^{q-1}(y) dy,
		\end{equation}
		and
		$$
		x_\epsilon=\frac{M_\epsilon^{-1} x}{\left|M_\epsilon^{-1} x\right|} \quad \text { and } \quad N_\epsilon=\epsilon M_\epsilon^{-1}=\left(\begin{array}{cc}
		I & 0 \\
		0 & \epsilon
		\end{array}\right) .
		$$
	\end{lemma}
	\begin{proof}
		Let
		$$
		u_\epsilon(x):=\left|M_\epsilon^{-1} x\right| \cdot h_\epsilon\left(\frac{M_\epsilon^{-1} x}{\left|M_\epsilon^{-1} x\right|}\right)
		$$
		By the invariance of the quantity $h_\epsilon^{n+1} \operatorname{det}\left(\nabla^2 h_\epsilon+h_\epsilon I\right)$ under linear transformations, see [\cite{WZ06}, Proposition 7.1] or formula (2.12) in \cite{LW2013}, we have
		\begin{equation}\label{det u}
		\operatorname{det}\left(\nabla^2 u_\epsilon+u_\epsilon I\right)(x)=\operatorname{det}\left(\nabla^2 h_\epsilon+h_\epsilon I\right)\left(\frac{M_\epsilon^{-1} x}{\left|M_\epsilon^{-1} x\right|}\right) \cdot \frac{\left(\operatorname{det} M_\epsilon^{-1}\right)^2}{\left|M_\epsilon^{-1} x\right|^{n+1}} .
		\end{equation}
		Since
		$$
		x_\epsilon=\frac{M_\epsilon^{-1} x}{\left|M_\epsilon^{-1} x\right|}=\frac{\left(\epsilon^{-1} x^{\prime}, x_n\right)}{\left|M_\epsilon^{-1} x\right|}=\frac{\left(x^{\prime}, \epsilon x_n\right)}{\left|N_\epsilon x\right|} .
		$$
		By \eqref{Minkowski}, we then have
		$$
		\begin{aligned}
		\operatorname{det}\left(\nabla^2 h_\epsilon+h_\epsilon I\right)\left(\frac{M_\epsilon^{-1} x}{\left|M_\epsilon^{-1} x\right|}\right) & =\frac{\left|x^{\prime}\right|^\alpha}{\left|N_\epsilon x\right|^\alpha} \cdot \frac{\left|\epsilon x_n\right|^\beta}{\left|N_\epsilon x\right|^\beta} \cdot\left(\frac{1}{\left|M_\epsilon^{-1} x\right|}\right)^{\gamma-\beta} \\
		& =\frac{\left|x^{\prime}\right|^\alpha}{\left|N_\epsilon x\right|^\alpha} \cdot \frac{\left|\epsilon x_n\right|^\beta}{\left|N_\epsilon x\right|^\beta} \cdot\left(\frac{\epsilon}{\left|N_\epsilon x\right|}\right)^{\gamma-\beta} \\
		& =\epsilon^{\gamma}\left|x^{\prime}\right|^\alpha\left|x_n\right|^\beta\left|N_\epsilon x \right|^{-\gamma-\alpha} .
		\end{aligned}
		$$
		Applying this into \eqref{det u}, we obtain
		$$
		\begin{aligned}
		\operatorname{det}\left(\nabla^2 u_\epsilon+u_\epsilon I\right)(x) & =\epsilon^{\gamma}\left|x^{\prime}\right|^\alpha\left|x_n\right|^\beta\left|N_\epsilon x\right|^{-\gamma-\alpha} \cdot \frac{\left(\epsilon^{1-n}\right)^2 \epsilon^{n+1}}{\left|N_\epsilon x\right|^{n+1}} \\
		& =\epsilon^{\gamma+3-n}\left|x^{\prime}\right|^\alpha\left|x_n\right|^\beta\left|N_\epsilon x\right|^{-\gamma-\alpha-n-1} .
		\end{aligned}
		$$
		By the definition of $u_\epsilon$, we have
		$$
		\begin{gathered}
		u_\epsilon(x)=\epsilon^{-1}\left|N_\epsilon x\right| \cdot h_\epsilon\left(x_\epsilon\right), \\
		\left(\bar{\nabla} u_\epsilon\right)(x)=M_\epsilon^{-T}\left(\bar{\nabla} h_\epsilon\right)\left(x_\epsilon\right)=\epsilon^{-1} N_\epsilon\left(\bar{\nabla} h_\epsilon\right)\left(x_\epsilon\right) 
		\end{gathered}
		$$
		then
		$$
		\begin{aligned}
		\rho_{K_{u_\epsilon},\bar{\nabla}u_{\epsilon}(x)}(y)
		&=\max\{\lambda\in \mathbb{R}:\lambda y\in K_{u_\epsilon}-\bar{\nabla}u_{\epsilon}(x)\}\\
		&=\max\{\lambda\in \mathbb{R}:\lambda y\in M_\epsilon^{-1}K_{h_\epsilon}-M_\epsilon^{-1}M_\epsilon \bar{\nabla}u_{\epsilon}(x)\}\\
		&=\max\{\lambda\in \mathbb{R}:M_\epsilon^{-1}M_\epsilon \lambda y\in M_\epsilon^{-1}(K_{h_\epsilon}-M_\epsilon \bar{\nabla}u_{\epsilon}(x))\}\\
		&=\rho_{K_{h_\epsilon},M_\epsilon( \bar{\nabla}u_{\epsilon}(x) )}(M_\epsilon(y))\\
		&=\rho_{K_{h_\epsilon},\bar{\nabla}h_{\epsilon}(x_\epsilon) )}(M_\epsilon(y))
		\end{aligned}
		$$
		Hence
		\begin{equation}\label{quermass}
		\begin{aligned}
		\widetilde{V}_{q-1}\left(K_{u_\epsilon},\bar{\nabla}u_{\epsilon}(x)\right)
		&=\frac{1}{n}\int_{\mathbb{S}^{n-1}}\left(\rho_{K_{u_\epsilon},\bar{\nabla}u_{\epsilon}(x)}(u)\right)^{q-1}d u\\
		&=\frac{1}{n}\int_{\mathbb{S}^{n-1}}\left(\left|M_\epsilon u\right|^{-1} \cdot \rho_{K_{h_\epsilon},\bar{\nabla}h_{\epsilon}(x_\epsilon)}\left(\frac{M_\epsilon (u)}{\left|M_\epsilon (u)\right|}\right)\right)^{q-1}d u\\
		&=\frac{\epsilon ^{2-q}}{n}\int_{\mathbb{S}^{n-1}}\left|N_\epsilon y\right|^{q-1-n} \rho_{K_{h_\epsilon}, \bar{\nabla}h_{\epsilon}(x_\epsilon) )}^{q-1}(y)d y,
		\end{aligned}
		\end{equation}
		where we apply the integration by substitution
		$$
		\begin{gathered}
		y=\frac{M_\epsilon (u)}{\left|M_\epsilon (u)\right|}=\frac{\epsilon u', u_n}{\left|M_\epsilon (u)\right|}\\
		u=\frac{M_\epsilon^{-1} y}{\left|M_\epsilon^{-1} y\right|}=\frac{\frac{1}{\epsilon}(y', \epsilon y_n)}{\left|M_\epsilon^{-1} y\right|}.
		\end{gathered}
		$$
		then
		$$
		\begin{aligned}
		u_\epsilon^{1-p}\widetilde{V}_{q-1}([u_\epsilon],\bar{\nabla} u_\epsilon)\mbox{det}(\nabla^2u_\epsilon+u_\epsilon)
		&=\epsilon^{p-q+\gamma-n+4}\left|N_\epsilon (x)\right|^{-\gamma-\alpha-n-p}h_\epsilon^{1-p}(x_\epsilon)\left|x^{\prime}\right|^\alpha\left|x_n\right|^\beta\\ 
		&\quad \frac{1}{n}\int_{\mathbb{S}^{n-1}}\left|N_\epsilon y\right|^{q-1-n} \rho_{K_{h_\epsilon},\bar{\nabla}h_{\epsilon}(x_\epsilon) }^{q-1}(y)d y,
		\end{aligned}
		$$
		Let $H_\epsilon(x):=\epsilon^{\frac{n-p-4-\gamma+q}{n-p+q-1}}u_\epsilon,$ we have
		$$
		\begin{aligned}
		H_\epsilon^{1-p}\widetilde{V}_{q-1}([H_\epsilon],\bar{\nabla} H_\epsilon)\mbox{det}(\nabla^2H_\epsilon+H_\epsilon I)
		&=\left|N_\epsilon (x)\right|^{-\gamma-\alpha-n-p}h_\epsilon^{1-p}(x_\epsilon)\left|x^{\prime}\right|^\alpha\left|x_n\right|^\beta\\ 
		&\quad \frac{1}{n}\int_{\mathbb{S}^{n-1}}\left|N_\epsilon y\right|^{q-1-n} \rho_{K_{h_\epsilon},\bar{\nabla}h_{\epsilon}(x_\epsilon) }^{q-1}(y)d y,
		\end{aligned}
		$$
	\end{proof}
	Since $|y'|\leqslant \left|N_\epsilon y\right|=\sqrt{|y'|^2+\epsilon|y_n|^2}\leqslant 1,q\leqslant 1+n$ we have 
	$$
	\begin{aligned}
	\frac{1}{n}\int_{\mathbb{S}^{n-1}}\left|N_\epsilon y\right|^{q-1-n} \rho_{K_{h_\epsilon},\bar{\nabla}h_{\epsilon}(x_\epsilon) }^{q-1}(y)d y
	&\geqslant \frac{1}{n}\int_{\mathbb{S}^{n-1}} \rho_{K_{h_\epsilon},\bar{\nabla}h_{\epsilon}(x_\epsilon) }^{q-1}(y)d y\\ 
	&\geqslant c.
	\end{aligned}
	$$
	The second inequality is due to the uniform boundness of $h_\epsilon,$ \eqref{h epsilon bound}. Indeed, Since 
	$$
	\int_{\mathbb{S}^{n-1}} \rho_{K_{h_\epsilon},\bar{\nabla}h_{\epsilon}(x_\epsilon) }^{q-1}(y)d y=\frac{1}{2} \int_{S^{n-1}} X_{K_{h_\epsilon}}(\bar{\nabla}h_{\epsilon}(x_\epsilon) , y)^{q-1} d y
	$$
	where the parallel $X$-ray of $K_{h_\epsilon}$ is the nonnegative function on $\mathbb{R}^n \times S^{n-1}$ defined by
	$$
	{K_{h_\epsilon}}(z , u)=|K_{h_\epsilon}\cap(z+\mathbb{R} u)|, \quad z \in \mathbb{R}^n, \quad u \in S^{n-1} .
	$$
	We can choose a Borel set $Z\subset \mathbb{S}^{n-1}$ satisfying $|Z|\geqslant c,$ here, $c$ is a universal constant independent with $\epsilon,$ and $\forall y\in Z$ we have $$X_{K_{h_\epsilon}}(\bar{\nabla}h_{\epsilon}(x_\epsilon) , y)\geqslant c$$ for some uniform constant $c$. Indeed, in dimension 2, since $K_{h_\epsilon}$ is pinched between two bounded balls, $\forall z\in \partial K_{h_\epsilon},$ there is a Borel set $Z\subset \mathbb{S}^{1}$ with $\arcsin(\frac{\sqrt{3}c}{2C})\leqslant |Z|\leqslant \frac{2\pi}{3}$ such that $y\in Z,X_{K_{h_\epsilon}}(z, y)\geqslant c,$ here, the constant $c(C)$ is the radius of the inner(outer) ball of $K_{\epsilon}.$ The higher dimensional case is analogous, just do a rotation. 
	
	Then combining with \eqref{h epsilon bound}, we have $ f_\epsilon(x)\geqslant c_1 \left|X^{\prime}\right|^\alpha\left|x_n\right|^\beta\left|N_\epsilon X\right|^{-\gamma-\alpha-n-p}.$ Hence
	$$
	\begin{aligned}
	\int_{\mathbb{S}^{n-1}}f_\epsilon(x)d x 
	&\geqslant c_1\int_{\mathbb{S}^{n-1}}\left|x^{\prime}\right|^\alpha\left|x_n\right|^\beta\left|N_\epsilon x\right|^{-\gamma-\alpha-n-p}dx\\
	&=2c_1\omega_{n-2}\int_{0}^{\frac{\pi}{2}}\sin\theta^\alpha \cos\theta^\beta\left(\sin\theta^2+\epsilon^2\cos\theta^2\right)^{-\frac{\gamma+\alpha+n+p}{2}}\sin\theta^{n-2} d\theta\\
	&\geqslant c_2 \int_{0}^{\frac{\pi}{4}}\theta^{\alpha+n-2} \left(\theta^2+\epsilon^2\right)^{-\frac{\gamma+\alpha+n+p}{2}}d\theta\\
	&+c_2 \int_{\frac{\pi}{4}}^{\frac{\pi}{2}}(\frac{\pi}{2}-\theta)^\beta\left(1+\epsilon^2(\frac{\pi}{2}-\theta)^2\right)^{-\frac{\gamma+\alpha+n+p}{2}} d\theta \\
	&\geqslant c_2 \int_{0}^{\frac{\pi}{4}}t^\beta\left(1+\epsilon^2 t^2\right)^{-\frac{\gamma+\alpha+n+p}{2}} dt\\
	&\geqslant c_2 \left\{\begin{array}{lll}
	\int_{0}^{\frac{\pi}{4}}t^\beta 2^{-\frac{\gamma+\alpha+n+p}{2}} dt & \text { when } 0\leqslant \gamma+\alpha+n+p \\
	\int_{0}^{\frac{\pi}{4}}t^\beta dt & \text { when } 0> \gamma+\alpha+n+p 
	\end{array}\right. \\
	&\geqslant c,
	\end{aligned}
	$$
	where the third inequality is due to the integration by substitution $t=\frac{\pi}{2}-\theta,$ and we throw the first item; the fourth inequality is because that $0<\epsilon<1/2,1 \leqslant 1+\epsilon^2 t^2<2.$\\
	On the other hand, by \eqref{h epsilon bound}, we have
	$$
	f_\epsilon(x)\leqslant C \left|X^{\prime}\right|^\alpha\left|x_n\right|^\beta\left|N_\epsilon X\right|^{-\gamma-\alpha-n-p}\frac{1}{n}\int_{S^{n-1}}\left|N_\epsilon y\right|^{q-1-n}\rho_{K_{h_\epsilon},\bar{\nabla}h_{\epsilon}(x_\epsilon)}^{q-1}(y) dy,
	$$
	first, we estimate the integral, since $|y'|\leqslant \left|N_\epsilon y\right|=\sqrt{|y'|^2+\epsilon|y_n|^2}\leqslant 1,$ and by \eqref{h epsilon bound} we have
	$$
	\begin{aligned}
	\int_{S^{n-1}}\left|N_\epsilon y\right|^{q-1-n}\rho_{K_{h_\epsilon},\bar{\nabla}h_{\epsilon}(x_\epsilon))}^{q-1}(y) dy
	&\leqslant C\int_{S^{n-1}}\left|N_\epsilon y\right|^{q-1-n}dy\\
	&\leqslant C 2\omega_{n-2}\int_{0}^{\frac{\pi}{2}}\left(\sin\theta^2+\epsilon^2\cos\theta^2\right)^{\frac{q-1-n}{2}}\sin\theta^{n-2} d\theta ,
	\end{aligned}
	$$
	Since 
	$$
	\begin{aligned}
	\int_{0}^{\frac{\pi}{2}}\left(\sin\theta^2+\epsilon^2\cos\theta^2\right)^{\frac{q-1-n}{2}}\sin\theta^{n-2} d\theta
	&\leqslant C\int_{0}^{\frac{\pi}{4}}\theta^{n-2} \left(\theta^2+\epsilon^2\right)^{(q-1-n)/2}d\theta\\
	&+C \int_{\frac{\pi}{4}}^{\frac{\pi}{2}}\left(1+\epsilon^2(\frac{\pi}{2}-\theta)^2\right)^{(q-1-n)/2} d\theta \\
	&\leqslant C \left[\int_{0}^{\epsilon}\theta^{n-2} \epsilon^{q-1-n}d\theta+\int_{\epsilon}^{\frac{\pi}{4}}\theta^{q-3} d\theta+\frac{\pi}{4}\right]\\
	&\leqslant C\left\{\begin{array}{lll}
	\frac{\epsilon^{q-2}}{n-1}+\frac{\pi}{4}+C & \text { when } q>2 \\
	\frac{1}{n-1}+\frac{\pi}{4}+|\log\epsilon| & \text { when } q=2\\
	c\epsilon^{q-2} & \text { when } q<2
	\end{array}\right. \\
	&\leqslant C\left\{\begin{array}{lll}
	C & \text { when } q>2 \\
	|\log\epsilon| & \text { when } q=2\\
	\epsilon^{q-2} & \text { when } q<2
	\end{array}\right. \\
	\end{aligned}
	$$
	Since we have $2<q\leqslant n+1,$ then
	$$
	\begin{aligned}
	f_\epsilon(x)&\leqslant C\left|x^{\prime}\right|^\alpha\left|x_n\right|^\beta\left|N_\epsilon X\right|^{-\gamma-\alpha-n-p}\\
	&\leqslant C\left\{\begin{array}{lll}
	\left|x^{\prime}\right|^\alpha\left|x_n\right|^\beta & \text { when } 0> \gamma+\alpha+n+p  \\
	\left|x^{\prime}\right|^{-\gamma-p-n}\left|x_n\right|^\beta & \text { when } 0\leqslant \gamma+\alpha+n+p\\
	\end{array}\right. \\
	\end{aligned}
	$$
	and $\|f_\epsilon\|_{L^1(\mathbb{S}^{n-1})}\geqslant c.$ When $0\leqslant \gamma+\alpha+n+p,$ we need $-1<\gamma<-1-\frac{p}{n+q-1},$ then $-\gamma-p-n$ satisfies the condition of $\alpha$ in Theorem \ref{B}. Hence, $f_\epsilon$ satisfies \eqref{f}.
	Now, we need to estimate the chord integral of $K_{H_\epsilon}.$
	\begin{lemma}\label{chordintegral 1}
		We have the estimate for the $q-$th chord integral of $K_{H_\epsilon}$ as follows:
		$$
		\begin{aligned}
		I_q(K_{H_\epsilon})
		&\leqslant C\left\{\begin{array}{lll}
		\epsilon ^{2-\frac{3+\gamma}{n-p+q-1}(q+n-1)} & \text { when } q>2 \\
		\epsilon ^{2-\frac{3+\gamma}{n-p+1}(n+1)}|\log\epsilon| & \text { when } q=2\\
		\epsilon ^{q-\frac{3+\gamma}{n-p+q-1}(q+n-1)} & \text { when } q<2
		\end{array}\right. \\
		\end{aligned}
		$$
	\end{lemma}
	\begin{proof}
		Since 
		$$
		I_q(K_{H_\epsilon})=\frac{2 q}{(n+q-1)\omega_n} \int_{\partial K_{H_\epsilon}} H_\epsilon(\nu(z))\widetilde{V}_{q-1}(K_{H_\epsilon}, \bar{\nabla}H_\epsilon(\nu(z))) \mathrm{d} \mathcal{H}^{n-1}(z), 
		$$
		and recall the definition of $H_\epsilon$ \eqref{H epsilon} and by caculation above, we have
		$$
		\begin{aligned}
		\widetilde{V}_{q-1}\left(K_{H_\epsilon},\bar{\nabla}H_{\epsilon}(x)\right)
		&=\epsilon^{\frac{n-p-4-\gamma+q}{n-p+q-1}(q-1)}
		\widetilde{V}_{q-1}\left(K_{u_\epsilon},\bar{\nabla}u_{\epsilon}(x)\right)\\
		&=\frac{\epsilon ^{2-q+\frac{n-p-4-\gamma+q}{n-p+q-1}(q-1)}}{n}\int_{\mathbb{S}^{n-1}}\left|N_\epsilon y\right|^{q-1-n} \rho_{K_{h_\epsilon}, \bar{\nabla}h_{\epsilon}(x_\epsilon) )}^{q-1}(y)d y\\
		&\leqslant C\epsilon ^{2-q+\frac{n-p-4-\gamma+q}{n-p+q-1}(q-1)}\left\{\begin{array}{lll}
		1 & \text { when } q>2 \\
		|\log\epsilon| & \text { when } q=2\\
		\epsilon^{q-2} & \text { when } q<2
		\end{array}\right. \\
		\end{aligned}
		$$
		and combining the definition of $u_\epsilon,$ we have
		$$
		\Vol(K_{H_\epsilon})=\epsilon^{\frac{n-p-4-\gamma+q}{n-p+q-1}n+1-n}\Vol(K_{h_\epsilon})
		$$
		Since $\Vol(K_{h_\epsilon})\leqslant C,$ then we have
		$$
		\Vol(K_{H_\epsilon})\leqslant C\epsilon^{\frac{n-p-4-\gamma+q}{n-p+q-1}n+1-n}.
		$$
		Now, we can estimate the chord integral
		$$
		\begin{aligned}
		I_q(K_{H_\epsilon})
		&\leqslant C\epsilon ^{2-q+\frac{n-p-4-\gamma+q}{n-p+q-1}(q-1)}\Vol(K_{H_\epsilon})\left\{\begin{array}{lll}
		1 & \text { when } q>2 \\
		|\log\epsilon| & \text { when } q=2\\
		\epsilon^{q-2} & \text { when } q<2
		\end{array}\right. \\
		&\leqslant C\left\{\begin{array}{lll}
		\epsilon ^{2-\frac{3+\gamma}{n-p+q-1}(q+n-1)} & \text { when } q>2 \\
		\epsilon ^{2-\frac{3+\gamma}{n-p+1}(n+1)}|\log\epsilon| & \text { when } q=2\\
		\epsilon ^{q-\frac{3+\gamma}{n-p+q-1}(q+n-1)} & \text { when } q<2
		\end{array}\right. \\
		\end{aligned}
		$$
		The proof of this lemma is completed.  
	\end{proof}
	\begin{remark}\label{H chord}
		When $q>2$, choose $-1<\gamma<-\frac{2p}{n+q-1}-1.$ Then 
		$$
		I_q(K_{H_\epsilon})\leqslant \epsilon ^{2-\frac{3+\gamma}{n-p+q-1}(q+n-1)}\rightarrow 0,\text{ as }\epsilon\rightarrow 0^+.
		$$
		When $q=2$, choose $-1<\gamma<-\frac{2p}{n+1}-1.$ Then 
		$$
		I_q(K_{H_\epsilon})\leqslant \epsilon ^{2-\frac{3+\gamma}{n-p+1}(n+1)}|\log\epsilon|\rightarrow 0,\text{ as }\epsilon\rightarrow 0^+.
		$$
		When $0<q<2$, $p\in(-\infty, -\frac{(2-q)(n+q-1)}{q}),$ choose $-1<\gamma<-\frac{qp}{n+q-1}+q-3.$ Then 
		$$
		I_q(K_{H_\epsilon})\leqslant \epsilon ^{q-\frac{3+\gamma}{n-p+q-1}(q+n-1)}\rightarrow 0,\text{ as }\epsilon\rightarrow 0^+.
		$$
	\end{remark}
	
	\section{Variational solution}
	
	This section is dedicated to solving the $L_p$ chord Minkowski problem in a variational method. Let $C_{re}^+(\mathbb{S}^{n-1})$ denote the set of rotationally symmetric, even, and positive continuous functions on $\mathbb{S}^{n-1}.$ Let $f\in C_{re}^+(\mathbb{S}^{n-1})$ satisfies \eqref{f}, and real numbers $\alpha,\beta$ are as in Theorem \ref{B}.
	We consider the maximization problem
	\begin{equation}\label{max}
	\sup_{h\in C_{re}^+(\mathbb{S}^{n-1})}\left\{\Phi_p(h):=\int_{\mathbb{S}^{n-1}}f(x)h(x)^p, I_q(K_h)=1\right\}.
	\end{equation}
	\begin{lemma}
		Let $h_i\in C_{re}^+(\mathbb{S}^{n-1})$ be a maximizing sequence, then there exists some uniform constant $C$ such that
		\begin{equation}\label{key estimate}
		\frac{1}{C}\leqslant h_i\leqslant C \text{  as  }i\rightarrow +\infty.
		\end{equation}
	\end{lemma}
	\begin{proof}
		Since $h_i\in C_{re}^+(\mathbb{S}^{n-1})$ is a maximizing sequence; that is $I_q(K_i)=1$ and 
		$$
		\lim_{i\rightarrow\infty}\Phi_p(h_i)=\sup_{h\in C_{re}^+(\mathbb{S}^{n-1})}\left\{\Phi_p(h): I_q(K_h)=1\right\}
		$$
		where $K_i$ is the convex body uniquely determined by $h_i.$ 
		By John' lemma, we have $\frac{1}{n}E_i\subset K_i\subset E_i.$ y. Since $K_i$ is rotationally symmetric and even, $E_i$ is also rotationally symmetric and even. Therefore, the centre of $E_i$ is at the origin and there exists a unique rotationally symmetric matrix $A_i$ of the form
		$$
		A_i=\left(\begin{array}{cccc}
		r_i a_i^{\frac{1}{n}} & & & \\
		& \ddots & & \\
		& & r_i a_i^{\frac{1}{n}} & \\
		& & & r_i a_i^{\frac{1-n}{n}}
		\end{array}\right), \text { where } r_i>0, a_i>0 \text { are constants, }
		$$
		such that
		$$
		\begin{aligned}
		& h_{E_i}(x)=\left|A_i x\right| \quad \text { on } \mathbb{S}^{n-1}, \\
		& \rho_{E_i}(u)=\left|A_i^{-1} u\right|^{-1} \text { on } \mathbb{S}^{n-1}\\
		& \Vol(E_i)=\kappa_n r_i^n.
		\end{aligned}
		$$
		
		Suppose to the contrary that
		$$
		\max_{x \in \mathbb{S}^{n-1}}h_i(x)\rightarrow +\infty \text{ or } \min_{x \in \mathbb{S}^{n-1}}h_i(x)\rightarrow 0,
		$$
		as $i\rightarrow +\infty.$ Since we have
		$$
		\frac{1}{n}\left|A_i x\right| \leq h_i(x) \leq\left|A_i X\right|, \quad \forall x \in \mathbb{S}^{n-1}
		$$
		which implies that
		$$
		r_i a_i^{\frac{1}{n}}\rightarrow +\infty \text{ or } r_i a_i^{\frac{1}{n}}\rightarrow 0 \text{ or } r_i a_i^{\frac{1-n}{n}}\rightarrow +\infty \text{ or } r_i a_i^{\frac{1-n}{n}}\rightarrow 0
		$$
		as $i\rightarrow +\infty,$ which must cause one of the four cases 
		$$
		r_i \rightarrow +\infty \quad r_i \rightarrow 0 \quad a_i \rightarrow +\infty \quad a_i\rightarrow 0.
		$$
		Since
		$$
		h_i(x) \geq \frac{1}{n} h_{E_i}(X)=\frac{1}{n}\left|A_i x\right| \text { on } \mathbb{S}^{n-1},
		$$
		and $p<0$, from the assumption of $f$ we have
		$$
		\begin{aligned}
		\Phi_p(h_i)&=\int_{\mathbb{S}^{n-1}}f(x)h_i(x)^p d x\\
		&\leqslant n^{-p}\int_{\mathbb{S}^{n-1}}f(x)\left|A_i x\right|^p d x\\
		&\leqslant C n^{-p}\int_{\mathbb{S}^{n-1}}\left|x^{\prime}\right|^\alpha\left|x_n\right|^\beta\left|A_i x\right|^p d x\\
		&:=C n^{-p}F(A_i)
		\end{aligned}
		$$
		we compute in the spherical coordinates as follows:
		$$
		\begin{aligned}
		F(A_i)= & r_i^p a_i^{\frac{p}{n}} \int_{\mathbb{S}^{n-1}}\left|x^{\prime}\right|^\alpha\left|x_n\right|^\beta\left(\left|x^{\prime}\right|^2+a_i^{-2} x_n^2\right)^{\frac{p}{2}} \mathrm{~d} x \\
		= & 2 r_i^p a_i^{\frac{p}{n}} \omega_{n-2} \int_0^{\frac{\pi}{2}} \sin ^\alpha \theta \cos ^\beta \theta\left(\sin ^2 \theta+a_i^{-2} \cos ^2 \theta\right)^{\frac{p}{2}} \sin ^{n-2} \theta \mathrm{d} \theta \\
		\leq & Cr_i^p a_i^{\frac{p}{n}} \int_0^{\frac{\pi}{4}} \theta^{\alpha+n-2}\left(\theta^2+a_i^{-2}\right)^{\frac{p}{2}} \mathrm{~d} \theta \\
		& +C r_i^p a_i^{\frac{p}{n}} \int_{\frac{\pi}{4}}^{\frac{\pi}{2}}\left(\frac{\pi}{2}-\theta\right)^\beta\left(1+a_i^{-2}\left(\frac{\pi}{2}-\theta\right)^2\right)^{\frac{p}{2}} \mathrm{~d} \theta .
		\end{aligned}
		$$
		
		Case 1:	When $a_i>3$. We can further estimate these following integrals:
		$$
		\begin{aligned}
		\int_0^{\frac{\pi}{4}}\left(\theta^2+a_i^{-2}\right)^{\frac{p}{2}} \theta^{\alpha+n-2} \mathrm{~d} \theta & \leq \int_0^{\frac{1}{a_i}} a_i^{-p} \theta^{\alpha+n-2} \mathrm{~d} \theta+\int_{\frac{1}{a_i}}^{\frac{\pi}{4}} \theta^p \theta^{\alpha+n-2} \mathrm{~d} \theta \\
		& =\frac{a_i^{-p-\alpha-n+1}}{\alpha+n-1}+\int_{\frac{1}{a_i}}^{\frac{\pi}{4}} \theta^{p+\alpha+n-2} \mathrm{~d} \theta \\
		& \leq C \begin{cases}1, & \text { if } p+\alpha+n-1>0, \\
		\log a_i, & \text { if } p+\alpha+n-1=0, \\
		a_i^{-p-\alpha-n+1}, & \text { if } p+\alpha+n-1<0,\end{cases}
		\end{aligned}
		$$
		and applying that $\beta>-1,$ we have 
		$$
		\begin{aligned}
		\int_{\frac{\pi}{4}}^{\frac{\pi}{2}}\left(\frac{\pi}{2}-\theta\right)^\beta\left(1+a_i^{-2}\left(\frac{\pi}{2}-\theta\right)^2\right)^{\frac{p}{2}} \mathrm{~d} \theta & =\int_0^{\frac{\pi}{4}} t^\beta\left(1+a_i^{-2} t^2\right)^{\frac{p}{2}} \mathrm{~d} t \\
		& \leq C \int_0^{\frac{\pi}{4}} t^\beta \mathrm{d} t \\
		& \leq C .
		\end{aligned}
		$$
		Then we obtain
		$$
		F(A_i) \leq C r_i^p a_i^{\frac{p}{n}} \begin{cases}1, & \text { if } p+\alpha+n-1>0, \\ \log a_i, & \text { if } p+\alpha+n-1=0, \\ a_i^{-p-\alpha-n+1}, & \text { if } p+\alpha+n-1<0 .\end{cases}
		$$
		
		Cace 2: When $1/3\leqslant a_i \leqslant 3.$ Since $\left|A_i x\right|=r_i a_i^{\frac{1}{n}}(\left|x^{\prime}\right|^2+a_i^{-2} x_n^2)^{\frac{1}{2}} ,$ by the assumption of $a_i$ and $\alpha>1-n, \beta>-1,$ we have
		$$
		C r_i\leqslant \left|A_i x\right|\leqslant \tilde{C}r_i,\quad \forall x\in \mathbb{S}^{n-1}.
		$$
		Then
		$$
		C r_i^p\leqslant F(A_i)\leqslant \tilde{C}r_i^p,\quad \forall x\in \mathbb{S}^{n-1}.
		$$
		
		Cace 3: When $a_i<1 / 3$, since we have
		$$
		F(A_i)=2 r_i^p a_i^{\frac{p}{n}-p} \omega_{n-2} \int_0^{\frac{\pi}{2}} \sin ^\alpha \theta \cos ^\beta \theta\left(a_i^2 \sin ^2 \theta+\cos ^2 \theta\right)^{\frac{p}{2}} \sin ^{n-2} \theta \mathrm{d} \theta .
		$$
		Then, applying $\alpha+n-1>0,$  we have
		$$
		\begin{aligned}
		F(A_i) & \leq C r_i^p a_i^{\frac{p}{n}-p}\left[\int_0^{\frac{\pi}{4}}  \theta^{\alpha+n-2}\left(a_i^2 \theta ^2 +1\right)^{\frac{p}{2}} \mathrm{~d} \theta+\int_{\frac{\pi}{4}}^{\frac{\pi}{2}}\left(\frac{\pi}{2}-\theta\right)^\beta\left(a_i^2+\left(\frac{\pi}{2}-\theta\right)^2\right)^{\frac{p}{2}} \mathrm{~d} \theta\right] \\
		& \leq C r_i^p a_i^{\frac{p}{n}-p}\left[1+\int_0^{\frac{\pi}{4}} t^\beta\left(a_i^2+t^2\right)^{\frac{p}{2}} \mathrm{~d} t\right] .
		\end{aligned}
		$$
		And
		$$
		\begin{aligned}
		\int_0^{\frac{\pi}{4}} t^\beta\left(a_i^2+t^2\right)^{\frac{p}{2}} \mathrm{~d} t & \leq \int_0^{a_i} t^\beta a_i^p \mathrm{~d} t+\int_{a_i}^{\frac{\pi}{4}} t^\beta t^p \mathrm{~d} t \\
		& =\frac{a_i^{p+\beta+1}}{\beta+1}+\int_{a_i}^{\frac{\pi}{4}} t^{\beta+p} \mathrm{~d} t \\
		& \leq C \begin{cases}1, & \text { if } \beta+p+1>0, \\
		|\log a_i|, & \text { if } \beta+p+1=0, \\
		a_i^{\beta+p+1,} & \text { if } \beta+p+1<0 .\end{cases}
		\end{aligned}
		$$
		Then we obtain
		$$
		F(A_i) \leq C r_i^p a_i^{\frac{p}{n}-p} \begin{cases}1, & \text { if } \beta+p+1>0, \\ |\log a_i|, & \text { if } \beta+p+1=0, \\ a_i^{\beta+p+1}, & \text { if } \beta+p+1<0 .\end{cases}
		$$
		\begin{claim}
			If one of the four caces $r_i \rightarrow +\infty \quad r_i \rightarrow 0 \quad a_i \rightarrow +\infty \quad a_i\rightarrow 0$ occurs, then $F(A_i)\rightarrow 0.$
		\end{claim}
		\begin{proof}
			Since $I_q(K_i)=1,$ and $\frac{1}{n}E_i\subset K_i\subset E_i,$ We have
			$$
			(\frac{1}{n})^{n+q-1}I_q(E_i)=I_q(\frac{1}{n}E_i)\leqslant I_q(K_i)\leqslant I_q(E_i).
			$$
			If we set $\Lambda=n^{q+n-1} $, then
			$$
			\Lambda^{-1} \leq I_q(E_i) \leq \Lambda
			$$
			By \ref{elliptic estimate}, when $q\in (1, n + 1)$ is not an integer, we have
			$$
			\begin{aligned}
			I_q(E_i)&=
			\left\{\begin{array}{lll}
			\left(r_i a_i^{\frac{1}{n}}\right)^{n+q-1}I_q\left(\left(r_i a_i^{\frac{1}{n}}\right)^{-1}E_i\right) & \text { when } a_i>1 \\
			\left(r_i a_i^{\frac{1-n}{n}}\right)^{n+q-1}I_q\left(\left(r_i a_i^{\frac{1-n}{n}}\right)^{-1}E_i\right) & \text { when } a_i\leqslant 1\\
			\end{array}\right. \\
			&\leqslant
			\left\{\begin{array}{lll}
			c_{q, m, n}\left(r_i a_i^{\frac{1}{n}}\right)^{n+q-1}a_i^{-2} & \text { when } a_i>1 \\
			c_{q, m, n}\left(r_i a_i^{\frac{1-n}{n}}\right)^{n+q-1}a_i^{n+q-2} & \text { when } a_i\leqslant 1\\
			\end{array}\right. \\
			&\leqslant c_{q, m, n}\left(r_i a_i^{\frac{1}{n}}\right)^{n+q-1}a_i^{-1}
			\end{aligned}
			$$
			where the last inequality is due to $a_i>1,$ which leads to $a_i^{-2}<a_i^{-1}.$\\
			When $q\in (1, n + 1)$ is an integer, we can choose a proper $q'\in (q,q+1),$ by \ref{chord integral ineq}, we have
			$$
			\begin{aligned}
			I_q(E_i)& \leqslant c(q',q)V(E_i)^{1-\frac{q-1}{q'-1}}I_{q'}(E_i)^{\frac{q-1}{q'-1}}\\
			& \leqslant \omega_n c(q',q)c_{q', q, n}\left(r_i a_i^{\frac{1}{n}}\right)^{n+q-1}a_i^{-1}.
			\end{aligned}
			$$
			On the other hand, 
			$$
			I_q(E_i)=\frac{1}{n \omega_n} \int_{S^{n-1}} \int_{E_i|u^{\bot}} X_{E_i}(x, u)^q \mathrm{~d} x \mathrm{~d} u, \quad q \geq 0 .
			$$
			where $E_i|u^{\bot}$ denotes the projection of $E_i$ onto $u^{\bot}.$ For $q>1,$ we have 
			$$
			I_q(E_i)\geqslant \frac{1}{n \omega_n} \int_{S^{n-1}} V(E_i)^q V_{n-1}(E_i|u^{\bot})^{1-q}du
			$$
			Indeed, H$\ddot{o}$lder inequality gives
			\begin{small}
				$$
				\frac{1}{V_{n-1}(E_i|u^{\bot})}\int_{E_i|u^{\bot}} X_{E_i}(x, u)^q \mathrm{~d} x\geqslant \left(\frac{1}{V_{n-1}(E_i|u^{\bot})}\int_{E_i|u^{\bot}} X_{E_i}(x, u) \mathrm{~d} x\right)^q=\left(\frac{V(E_i)}{V_{n-1}(E_i|u^{\bot})}\right)^q.
				$$
			\end{small}
			Recall that $V(E_i)=\kappa_n r_i^n.$ \\
			If $a_i>1$, $r_i a_i^{\frac{1}{n}}>r_i a_i^{\frac{1-n}{n}},$ it follows that $V_{n-1}(E_i|u^{\bot})\leqslant \kappa_{n-1} \left(r_i a_i^{\frac{1}{n}}\right)^{n-1}.$\\
			Hence,
			$$
			\begin{aligned}
			I_q(E_i)&\geqslant \frac{\kappa_{n-1}^{1-q}}{n\kappa_n^{1-q}}\left(r_i a_i^{\frac{1}{n}}\right)^{n-1}\left(r_i a_i^{\frac{1-n}{n}}\right)^{q}\\
			&=\frac{\kappa_{n-1}^{1-q}}{n\kappa_n^{1-q}}\left(r_i a_i^{\frac{1}{n}}\right)^{q+n-1} a_i^{-q}
			\end{aligned}
			$$
			If $a\leqslant 1$, $r_i a_i^{\frac{1}{n}}\leqslant r_i a_i^{\frac{1-n}{n}},$ it follows that $V_{n-1}(E_i|u^{\bot})\leqslant \kappa_{n-1} \left(r_i a_i^{\frac{1}{n}}\right)^{n-2}r_i a_i^{\frac{1-n}{n}}.$\\
			Hence,
			$$
			I_q(E_i)\geqslant \frac{\kappa_{n-1}^{1-q}}{n\kappa_n^{1-q}}\left(r_i a_i^{\frac{1}{n}}\right)^{q+n-1} a_i^{-1}.
			$$
			In conlusion, 
			$$
			\begin{gathered}
			\left(r_i a_i^{\frac{1}{n}}\right)^{q+n-1} a_i^{-q}\leqslant C,\left(r_i a_i^{\frac{1}{n}}\right)^{q+n-1} a_i^{-1}\geqslant c \quad a>1\\
			c \leqslant \left(r_i a_i^{\frac{1}{n}}\right)^{q+n-1} a_i^{-1}\leqslant C \quad a\leqslant 1
			\end{gathered}
			$$
			where the constants $c,C$ only depend on $n,q.$
			Therefore, we shall prove the claim case by case as follows:\\
			In case 1: $a>3,$ then $\left(r_i a_i^{\frac{1}{n}}\right)^p\leqslant c a^{\frac{p}{n+q-1}}.$ Then we obtain
			$$
			F(A) \leq C \begin{cases}a^{\frac{p}{n+q-1}}, & \text { if } p+\alpha+n-1>0, \\ a^{\frac{p}{n+q-1}}\log a, & \text { if } p+\alpha+n-1=0, \\ a^{\frac{p}{n+q-1}-p-\alpha-n+1}, & \text { if } p+\alpha+n-1<0 .\end{cases}
			$$
			By the assumptions in Theorem \ref{B}, we observe that the power of $a$ is negative. If one of $r_i \rightarrow +\infty \quad r_i \rightarrow 0 \quad a_i \rightarrow +\infty$ occurs, then $F(A_i)\rightarrow 0.$\\
			In case 2: If $1/3\leqslant a\leqslant 1,$ then $c a^{\frac{1}{n+q-1}-\frac{1}{n}} \leqslant r_i\leqslant C a^{\frac{1}{n+q-1}-\frac{1}{n}} ;$ If $1\leqslant a\leqslant 3,$ then $c a^{\frac{1}{n+q-1}-\frac{1}{n}} \leqslant r_i\leqslant C a^{\frac{q}{n+q-1}-\frac{1}{n}} .$ That is $c \leqslant r_i\leqslant C.$ Then we obtain
			$$
			F(A)\leqslant \tilde{C}a^{\frac{p}{n+q-1}-\frac{p}{n}}\leqslant C.
			$$
			None of the four caces $r_i \rightarrow +\infty \quad r_i \rightarrow 0 \quad a_i\rightarrow +\infty \quad a_i\rightarrow 0$ can occur.\\
			In case 3: $a<1/3,$ then $\left(r_i a_i^{\frac{1}{n}}\right)^pa_i^{-p}\leqslant c a^{\frac{p}{n+q-1}-p}.$ Then we obtain
			$$
			F(A) \leq C \begin{cases}a^{\frac{p}{n+q-1}-p}, & \text { if } \beta+p+1>0, \\ a^{\frac{p}{n+q-1}-p}|\log a|, & \text { if } \beta+p+1=0, \\ a^{a^{\frac{p}{n+q-1}-p}+\beta+p+1}, & \text { if } \beta+p+1<0 .\end{cases}
			$$
			By the assumptions in Theorem \ref{B}, we observe that the power of $a$ is positive. If one of $r_i \rightarrow +\infty \quad r_i \rightarrow 0 \quad a_i\rightarrow 0$ occurs, then $F(A_i)\rightarrow 0.$
		\end{proof}	
		This claim means that if one of the four cases $r_i \rightarrow+\infty, r_i \rightarrow 0, a_i \rightarrow+\infty, a_i \rightarrow 0$ occurs, we have that
		$$
		\Phi_p(h_i) \rightarrow 0, \quad \text { as } \quad i \rightarrow+\infty
		$$
		However, taking $h\equiv r_0$ for some $r_0$ such that $I_q(B_{r_0})=1.$ Indeed, since $I_q(B_{r_0})=r_0^{n+q-1}I_q(B_1),$ we can choose $r_0=\frac{\omega_q}{2^q\omega_n\omega_{n+q-1}}.$ Hence, we have
		\begin{eqnarray}
		\sup_{h\in C_{re}^+(\mathbb{S}^{n-1})}\left\{\Phi_p(h): I_q(K_h)=1\right\}
		&\geqslant& \int_{\mathbb{S}^{n-1}}f r_0^p\\
		&\geqslant& C_1\|f\|_{L^1\left(\mathbb{S}^{n-1}\right)}>0,
		\end{eqnarray}
		which is a contradiction. Therefore, $\left\{h_i\right\}$ has uniformly positive upper and lower bounds.
	\end{proof}
	\begin{lemma}\label{exi}
		The maximization problem has a solution $h.$
	\end{lemma}
	\begin{proof}
		Let $h_i\in C_{re}^+(\mathbb{S}^{n-1})$ be a maximizing sequence; that is $I_q(K_i)=1$ and 
		$$
		\lim_{i\rightarrow\infty}\Phi_p(h_i)=\sup_{h\in C_{re}^+(\mathbb{S}^{n-1})}\left\{\Phi_p(h): I_q(K_h)=1\right\}
		$$
		where $K_i$ is the convex body uniquely determined by $h_i.$ By \ref{key estimate} we have $ c\leqslant h_i\leqslant C$ as $i$ big enough. By the Blaschke selection theorem, there is a subsequence of $\left\{h_i\right\}$ that uniformly converges to some support function $h$ on $\mathbb{S}^{n-1}$. Note that $h$ is also rotationally symmetric and even on $\mathbb{S}^{n-1}$, satisfying \eqref{key estimate}, $I_q(K_h)=1$, and
		$$
		\Phi_p(h)=\lim _{i \rightarrow+\infty}\Phi_p(h_i)=\sup_{h\in C_{re}^+(\mathbb{S}^{n-1})}\left\{\Phi_p(h): I_q(K_h)=1\right\}
		$$
		Hence, $h$ is a solution to the maximization problem. 
	\end{proof}
	\begin{theorem}
		Let $h$ be the maximizer obtained from lemma \ref{exi}. Then $h$ is a generalized solution of 
		$$
		\mbox{det}(\nabla^2h+hI) =\frac{fh^{p-1}}{C\widetilde{V}_{q-1}([h],\bar{\nabla} h)},
		$$
		where $C=\frac{2q}{(n+q-1)\omega_n}\int fh^p.$
	\end{theorem}
	\begin{proof}
		Let $h$ be the maximizer obtained from lemma \ref{exi} , we have that $1/C \leqslant h\leqslant C$ by \eqref{key estimate}. For any given rotationally symmetric and even $g\in C(\mathbb{S}^{n-1}),$ let
		$$
		K_t=\left\{x\in\mathbb{R}^n:x\cdot u\leqslant(h+tg)(u)\right\} 
		$$
		for sufficiently small $\delta>0$ such that $|t|<\delta$, we have $ h+tg>0,$ and $h_t\in C_{re}^+(\mathbb{S}^{n-1}),$ where $h_t$ is the support function of $K_t.$ Note that $h_0=h,K_0=K.$ 
		Let $\lambda(t)=I_q(K_t)^{-\frac{1}{n+q-1}},$ then $I_q(\lambda(t)K_t)=1,$ and $\lambda'(0)=-\frac{1}{n+q-1}\int_{\mathbb{S}^{n-1}}g(v)dF_q(K,v).$ Denote 
		$$
		\Psi_p(t)=\Phi_{p}(\lambda(t)h_t).
		$$
		As $h$ is a maximizer, the function $t\mapsto \Psi_p(t)$ attains its maximum at $t=0$. However, $ h_t$ may not be differentiable at $t=0.$ Let
		$$
		\psi_p(t)=\Phi_{p}(\lambda(t)(h+tg)).
		$$
		Since $K_t$ is the Wulff shape of $h+tg$, we have $h_t\leqslant h+tg.$ Since $f$ nonnegative, $ \lambda(t)>0,p<0$ and 
		$$
		\lambda(t)h_t \leqslant \lambda(t)(h+tg),
		$$
		thus we have
		$$
		\Psi_p(0)\geqslant \Psi_p(t)\geqslant \psi_p(t).
		$$
		Since $\Psi_p(0)=\psi_p(0),$ $\psi_p(t)$ also attains its maximum at $t=0.$ Hence we have
		$$
		\lim_{t_k\rightarrow 0^+}\frac{\psi_p(t_k)-\psi_p(0)}{t_k} \leqslant 0,
		$$
		for any convergent subsequence $\{t_k\}.$
		Therefore, we have
		$$
		\int_{S^{n-1}} ph(x)^{p-1}f(x)\left(\lambda^{\prime}(0) h+g\right) \leqslant 0 ,
		$$
		since $p<0$ and we can also replace $g$ by $-g,$ it follows that
		$$
		\int f h(x)^{p-1} g = \frac{1}{n+q-1}\int fh(x)^{p}\int g d F_q(K,\cdot) .
		$$
		Since $g\in C(\mathbb{S}^{n-1})$ is arbitrary, we conclude that 
		$$
		\mbox{det}(\nabla^2h+hI) =\frac{fh^{p-1}}{C\widetilde{V}_{q-1}([h],\bar{\nabla} h)},
		$$
		where $C=\frac{2q}{(n+q-1)\omega_n}\int fh^p.$
	\end{proof}
	Now, we can prove Theorem \ref{B}. Let $h$ be the maximizer obtained from lemma \ref{exi}, after a suitable scaling $\alpha h$ solves \eqref{MAeq} with $\alpha =C^{\frac{1}{n+q-p-1}}$. 
	Denote $h_0=\alpha h,$ $K_0$ is its corresponding convex body. We already know that $h_0$ is a generalized solution to \eqref{MAeq}. Now we want to estimate the chord integral of $K_0.$ 
	Since $I_q(K_0)=C^{\frac{n+q-1}{n+q-p-1}}I_q(K)=C^{\frac{n+q-1}{n+q-p-1}},$ we only need to estimate the bound of $C$ from below. Applying \eqref{key estimate}, $h$ has uniform bound from above. Hence, we have
	$$
	\begin{aligned}
	C &=\frac{2q}{(n+q-1)\omega_n}\int fh^p\\
	&\geqslant c \frac{2q}{(n+q-1)\omega_n}\|f\|_{L^1(\mathbb{S}^{n-1})}\\
	&> 0.
	\end{aligned}
	$$
	The proof of Theorem \ref{B} is completed.
	
	\section{Proof to Theorem \ref{A}}
	In this section, we aim to prove Theorem \ref{A}. It's a result which follows easily from Section 3 and Section 4.
	\begin{proof}
		When $p<0$ and $2<q<n+1,$ for any given $\epsilon \in (0,1/2),$ let 
		\begin{equation}
		f_\epsilon(x):=h_\epsilon\left(x_\epsilon\right)^{1-p}\left|X^{\prime}\right|^\alpha\left|x_n\right|^\beta\left|N_\epsilon X\right|^{-\gamma-\alpha-n-p}\frac{1}{n}\int_{S^{n-1}}\left|N_\epsilon y\right|^{q-1-n}\rho_{K_{h_\epsilon},\bar{\nabla}h_{\epsilon}(x_\epsilon)}^{q-1}(y) dy,
		\end{equation}
		where $h_\epsilon,x_\epsilon,\alpha,\beta,\gamma,N_\epsilon$ are as in Section 3. And we already know $H_\epsilon$ is a generalized solution to equation \eqref{MAeq} with $f$ replaced by $f_\epsilon.$ From the analysis of $f_\epsilon$ in Section 3, we know $f_\epsilon$ satisfies the condition \eqref{f}. Hence, by Theorem \ref{B}, we obtain a variational solution $h_0$, which is also a generalized solution to equation \eqref{MAeq} with $f$ replaced by $f_\epsilon,$ and the $q$-th chord integral of its corresponding convex body has a uniform positive bound from below. That is
		$$
		I_q(K_{h_0})\geqslant c>0,
		$$
		where $c$ depends only on $n,p,q,\alpha,\beta$ and $\gamma$ and is independent with $\epsilon.$\\
		However, by Remark \ref{H chord}, we can choose $\epsilon$ small enough such that 
		$$
		I_q(K_{H_\epsilon})<\frac{c}{2},
		$$
		here $c$ is the same as above.
		Hence, $H_\epsilon$ and $h_0$ are different solutions to equation \eqref{MAeq}.\\
		Therefore, the proof of Theorem \ref{A} is completed.
	\end{proof}

\end{document}